\documentclass[11pt]{amsart}
\usepackage{amssymb}

\title[Galois embedding of K3 surface]{Galois embedding of K3 surface \\ -- {\footnotesize abelian case} --}
\author[Hisao Yoshihara]{}

\newtheorem{theorem}{Theorem}[section]
\newtheorem{lemma}[theorem]{Lemma}
\newtheorem{proposition}[theorem]{Proposition}

\newtheorem{corollary}[theorem]{Corollary}
\newtheorem{claim}[theorem]{Claim}
\newtheorem{rep}{Representation}

\theoremstyle{definition}
\newtheorem{definition}[theorem]{Definition}
\newtheorem{example}[theorem]{Example}
\theoremstyle{remark}
\newtheorem{remark}[theorem]{Remark}

\newenvironment{namelist}[1]{%
\begin{list}{}
  {
   \settowidth{\labelwidth}{#1}
   \setlength{\leftmargin}{2.5\labelwidth}}
}{%
\end{list}}

\begin{document}
\maketitle

\begin{center}

{\sc Hisao Yoshihara}\\
\medskip
{\small{\em Department of Mathematics, Faculty of Science, Niigata University,\\
Niigata 950-2181, Japan}\\
E-mail:{\tt yosihara@math.sc.niigata-u.ac.jp}}
\end{center}

\bigskip

\begin{abstract}
We study Glois embeddings of $K3$ surfaces in the case where the Galois groups are abelian. 
We show several properties of $K3$ surfaces concerning the Galois embeddings. In particular, if the Galois group $G$ is abelian, then 
 $G \cong \mathbb Z/4\mathbb Z$, $\mathbb Z/6\mathbb Z$ or $(\mathbb Z/2\mathbb Z)^{\oplus 3}$ and $S$ is a smooth complete intersection 
of hypersurfaces.
Further, we state the detailed structure of such surfaces. 
\end{abstract}

\bigskip

\section{Introduction}
The purpose of this article is to study Galois embeddings of $K3$ surfaces, where the Galois groups 
are abelian. The non-abelian case will be treated later.   
Before going into the study on $K3$ surfaces, 
 we recall the definition of Galois embeddings of algebraic varieties and their properties. 

Let $k$ be the ground field of our discussion, we assume it to be the field of complex numbers, however 
most results hold also for an algebraically closed field of characteristic zero.  
Let $V$ be a nonsingular projective algebraic variety of dimension $n$ with a very ample divisor $D$, we 
denote this by a pair $(V, D)$. Let $f=f_D:V \hookrightarrow {\mathbb P}^N$ be the embedding of $V$ associated
 with the complete linear system $|D|$, where $N+1=\dim{\mathop{\mathrm{H^0}}\nolimits}(V,\ {\mathcal O}(D))$. Suppose that $W$ is a linear
 subvariety of ${\mathbb P}^N$ satisfying $\dim W=N-n-1$ and $W \cap f(V)=\emptyset$. Consider the projection $\pi_W$ 
from $W$ to $\mathbb P^n$, $\pi_W : {\mathbb P}^N \dashrightarrow \mathbb P^n$. Restricting $\pi_W$ onto $f(V)$, we get a surjective morphism $\pi=\pi_W \cdot f : V \longrightarrow \mathbb P^n$. 

Let $K=k(V)$ and $K_0=k(\mathbb P^n)$ be the function fields of $V$ and $\mathbb P^n$ respectively. The morphism $\pi$ 
induces a finite extension of fields ${\pi}^* : K_0 \hookrightarrow K$ of degree $d=\deg f(V)=D^n$, which is the
 self-intersection number of $D$. 
We denote by $K_W$ the Galois closure of this extension and by $G_W=Gal(K_W/K_0)$ the Galois group of $K_W/K_0$. 
By \cite{ha} we see that $G_W$ is isomorphic to the monodromy group of the covering $\pi : V \longrightarrow \mathbb P^n$.
Let $V_W$ be the $K_W$-normalization of $V$ (cf. \cite[Ch.2]{ii}). 
Note that $V_W$ is determined uniquely by $V$ and $W$. 

\begin{definition}\label{1}
In the above situation we call $G_W$ and $V_W$ the Galois group and the Galois closure variety at $W$ 
respectively (cf. \cite{y6}). 
If the extension $K/K_0$ is Galois, then we call $f$ and $W$ a Galois embedding and a Galois subspace for
 the embedding respectively. 
\end{definition}

\begin{definition}\label{3}
A nonsingular projective algebraic variety $V$ is said to have a Galois embedding if there exist a very ample 
divisor $D$ satisfying that the embedding associated with $|D|$ has a Galois subspace. In this case the pair $(V,D)$ 
is said to define a Galois embedding.
\end{definition}

If $W$ is the Galois subspace and $T$ is a projective transformation of ${\mathbb P}^N$, then $T(W)$ is a Galois
 subspace of the embedding $T \cdot f$. Therefore the existence of Galois subspace does not depend on the choice 
of the basis giving the embedding. 

\begin{remark}\label{31}
If a smooth variety $V$ exists in a projective space, then by taking a linear subvariety, we can define a Galois subspace and Galois group 
similarly as above. 
Suppose that $V$ is not normally embedded and there exists a linear subvariety $W$ such that the projection 
$\pi_W$ induces a Galois extension of fields . Then, taking $D$ as a hyperplane section of $V$ in the embedding,
 we infer readily that $(V,D)$ defines a Galois embedding with the same Galois group in the above sense.
\end{remark}

By this remark, for the study of Galois subspaces, it is sufficient to consider the case where $V$ is normally embedded. 

\medskip

We have studied Galois subspaces and Galois groups for hypersurfaces in \cite{y1}, \cite{y2} and \cite{y3} and space 
curves in \cite{y5} and \cite{y7}. 
The method introduced in \cite{y6} is a generalization of the ones in these studies. 

\medskip

Hereafter we use the following notation and convention:

\begin{namelist}{}
\item[$\cdot$]${\rm Aut}(V)$ : the automorphism group of a variety $V$
\item[$\cdot$]$|G|$ : the order of a group $G$
\item[$\cdot$]$\sim$ : the linear equivalence of divisors
\item[$\cdot$]${\bf 1}_m$ : the unit matrix of size $m$
\item[$\cdot$]$[\alpha_1, \ldots, \alpha_m]$ : the diagonal matrix with entries $\alpha_1, \ldots, \alpha_m$ 
\end{namelist}

\bigskip

The organization of this article is as follows: In Section 2 we review the results of Galois embeddings, 
which will be used in the sequel. We devote the remainder sections to the study of the Galois embedding of $K3$ surfaces.

\section{Results on Galois embeddings}

We state several properties concerning Galois embedding without proofs, for the details see \cite{y6}. 
By definition, if $W$ is a Galois subspace, then each element $\sigma$ of $G_W$ is an automorphism of $K=K_W$ over $K_0$. Therefore it induces a birational transformation of $V$ over $\mathbb P^n$. This implies that $G_W$ can be viewed as a subgroup of ${\rm Bir}(V/\mathbb P^n)$, the group of birational transformations of $V$ over $\mathbb P^n$. Further we can say the following:

\begin{rep}\label{a1}
Each birational transformation belonging to $G_W$ turns out to be regular on $V$, hence we have a faithful representation 
$$
\alpha :  G_W \hookrightarrow {\rm Aut}(V). \eqno(1)
$$
\end{rep}

\medskip

Therefore, if the order of ${\rm Aut}(V)$ is smaller than the degree $d$, then $(V,D)$ cannot define a Galois embedding. 
In particular, if ${\rm Aut}(V)$ is trivial, then $V$ has no Galois embedding. 
On the other hand, in case $V$ has infinitely many automorphisms, we have examples such that there exist infinitely many distinct Galois embeddings, see Example 4.1 in \cite{y6}. 
\medskip

When $(V,D)$ defines a Galois embedding, we often identify $f(V)$ with $V$. Let $H$ be a hyperplane of ${\mathbb P}^N$ containing $W$. Let $D'$ be  the intersection divisor of $V$ and $H$. 
Since $D' \sim D$ and ${\sigma}^*(D')=D'$, for any $\sigma \in G_W$, we see that $\sigma$ induces an automorphism of ${\mathop{\mathrm{H^0}}\nolimits}(V, {\mathcal O}(D))$. 
This implies the following. 

\begin{rep}\label{18}
We have a second faithful representation 

$$\beta : G_W \hookrightarrow PGL(N, {\mathbb C}).\eqno(2)$$

\end{rep}

\medskip
In the case where $W$ is a Galois subspace we identify $\sigma \in G_W$ with $\beta(\sigma) \in PGL(N, {\mathbb C})$ hereafter. 
Since $G_W$ is a finite subgroup of ${\rm Aut}(V)$, we can consider the quotient $V/G_W$ and let $\pi_G$ be the quotient morphism, $\pi_G : V \longrightarrow V/G_W$.

\begin{proposition}\label{15}
If $(V,D)$ defines a Galois embedding with the Galois subspace $W$ such that the projection is $\pi_W : {\mathbb P}^N \dashrightarrow \mathbb P^n$, then there exists an isomorphism $g:V/G_W \longrightarrow \mathbb P^n$ satisfying $g \cdot \pi_G = \pi$. Hence the projection $\pi$ turns out to be a finite morphism and the fixed loci of $G_W$ consist of only divisors.
\end{proposition}

Therefore, $\pi$ is a Galois covering in the sense of Namba \cite{na}.
We have a criterion that $(V,D)$ defines a Galois embedding. 

\begin{theorem}\label{8}
The pair $(V, D)$ defines a Galois embedding if and only if the following conditions hold{\rm :}
\begin{enumerate}
\item There exists a subgroup $G$ of ${\rm Aut}(V)$ satisfying that $|G|=D^n$.
\item There exists a $G$-invariant linear subspace ${\mathcal L}$ of ${\mathop{\mathrm{H^0}}\nolimits}(V, {\mathcal O}(D))$ of dimension $n+1$ such that, for any $\sigma \in G$, the restriction ${\sigma}^*|_{\mathcal L}$ is a multiple of the identity. 
\item The linear system ${\mathcal L}$ has no base points. 
\end{enumerate}
\end{theorem}

It is easy to see that $\sigma \in G_W$ induces an automorphism of $W$, hence we obtain another representation of $G_W$ as follows. 
Take a basis $\{f_0, f_1, \ldots, f_N  \}$ of ${\mathop{\mathrm{H^0}}\nolimits}(V, {\mathcal O}(D))$ satisfying that $\{f_0, f_1, \ldots, f_n  \}$ is a basis of ${\mathcal L}$ in Theorem \ref{8}. Then we have the representation  

\[
\newcommand{\bg}{%
 \family{cmr}\size{20}{12pt}\selectfont}
 \newcommand{\bigzerol}{\smash{\hbox{\bg 0}}}
 \newcommand{\bigzerou}{%
  \smash{\lower1.7ex\hbox{\bg 0}}}
\beta_1(\sigma)=\begin{pmatrix}\lambda_{\sigma} & & & \vdots &  \\
 & \ddots & & \vdots & {\bf *} \\
 & & \lambda_{\sigma} & \vdots & \\
 \cdots & \cdots & \cdots & \vdots & \cdots  \\
& {\bf 0} & & \vdots & M'_{\sigma}   
\end{pmatrix}. \eqno(3) 
\]
Since the representation is completely reducible, we get another 
representation using a direct sum decomposition:
\[
\beta_2(\sigma)=\lambda_{\sigma} \cdot {\bf 1}_{n+1}\oplus M'_{\sigma}. 
\]
Thus we can define
\[
\gamma(\sigma)=M'_{\sigma} \in PGL(N-n-1, {\mathbb C}). 
\]
Therefore $\sigma$ induces an automorphism on $W$ given by  $M'_{\sigma}$. 

\begin{rep}\label{19}
We get a third representation 

$$
\gamma : G_W \longrightarrow PGL(N-n-1, {\mathbb C}).\eqno(4)
$$

\end{rep}

\bigskip

Let $G_1$ and $G_2$ be the kernel and image of $\gamma$ respectively.

\begin{theorem}\label{a3}
We have an exact sequence of groups 
\[
1 \longrightarrow G_1 \longrightarrow G \stackrel{\gamma}{\longrightarrow} G_2 \longrightarrow 1, 
\]
where $G_1$ is a cyclic group.
\end{theorem}

\begin{corollary}\label{a4}
If $N=n+1$, i.e., $f(V)$ is a hypersurface, then $G$ is a cyclic group.
\end{corollary}

This assertion has been obtained in \cite{y3}. 
Moreover we have another representation.

Suppose that $(V, D)$ defines a Galois embedding and let $G$ be a Galois group for some Galois subspace $W$. Then, 
take a general hyperplane $W_1$ of $\mathbb P^n$ and put $V_1={\pi}^*(W_1)$. The divisor $V_1$ has the 
following properties{\rm :}
\begin{namelist}{(iii)}
\item[\rm{(i)}]If $n \geq 2$, then $V_1$ is a smooth irreducible variety. 
\item[\rm{(ii)}]$V_1 \sim D$. 
\item[\rm{(iii)}]${\sigma}^*(V_1)=V_1$ for any  $\sigma \in G$. 
\item[\rm{(iv)}] $V_1/G$ is isomorphic to $W_1$.
\end{namelist}
Put $D_1=V_1 \cap H_1$, where $H_1$ is a general hyperplane of ${\mathbb P}^N$. Then $(V_1, D_1)$ 
defines a Galois embedding with the Galois group $G$ (cf. Remark \ref{31}). 
Iterating the above procedures, we get a sequence of pairs $(V_i, D_i)$ such that 
\[
(V,D) \supset (V_1, D_1) \supset \cdots \supset (V_{n-1}, D_{n-1}). 
\]
These pairs satisfy the following properties:
\begin{enumerate}
\item[(a)]$V_i$ is a smooth subvariety of $V_{i-1}$, which is a hyperplane section of $V_{i-1}$, 
where $D_i=V_{i+1}$, $V=V_0$ and $D=V_1$ ($1 \le i \le n-1$). 
\item[(b)] $(V_i, D_i)$ defines a Galois embedding with the same Galois group $G$.
\end{enumerate}
In particular, letting $C$ be the curve $V_{n-1}$, we get the following fourth representation. 

\bigskip

\begin{rep}\label{16}
We have a fourth faithful representation 
$$
\delta : G_W \hookrightarrow {\rm Aut}(C), \eqno(5)
$$
where $C$ is a smooth curve in $V$ given by $V \cap L$ such that $L$ is a general linear subvariety of
 ${\mathbb P}^N$ with dimension $N-n+1$ containing $W$. 
\end{rep}

Note that in some cases there exist several Galois subspaces and Galois groups for one embedding (see, for example \cite{y7}). 
Generally we have the following.

\begin{proposition}\label{a5}
Suppose that $(V,D)$ defines a Galois embedding and let $W_i$ {\rm (}$i=1,2${\rm )} be Galois subspaces such
 that $W_1 \ne W_2$. Then $G_1 \ne G_2$ in ${\rm Aut}(V)$, where $G_i$ is the Galois group at $W_i$. 
\end{proposition}

\begin{corollary}\label{a6}
If $V$ is a smooth projective algebraic variety of general type, then there are at most finitely many Galois subspaces.
\end{corollary}

\begin{remark}\label{a7}
It may happen that there exist infinitely many Galois subspaces for one embedding if the Kodaira dimension 
of $V$ is small. For example, if 
$V={\mathbb P}^1$ and $\deg D=3$, i.e., $f(V)$ is a twisted cubic, then the Galois lines form two dimensional 
locally closed subvariety of the Grassmannian ${\mathbb G}(1,3)$, parametrizing lines in projective three space (cf. \cite{y5}). 
\end{remark}

\section{$K3$ Surfaces}
We apply the methods developed in the previous sections to the study of $K3$ surfaces. For each abelian surface 
with a Galois embedding, we have studied in detail in \cite{y6}. 
In particular, we have given the complete list of the complex representation of every possible group and shown that the surface is isogenous to 
the square of an elliptic curve. 

A curve and surface will mean a nonsingular projective algebraic curve and surface respectively. 
In addition to the notation listed in Section 1, we use the following hereafter:   

\begin{namelist}{}
\item[$\cdot$]$\langle a_1, \cdots, a_m \rangle$ : the subgroup generated by $a_1, \cdots, a_m$ 
\item[$\cdot$]$Z_m$ : the cyclic group of order $m$
\item[$\cdot$]$e_m := \exp(2\pi \sqrt{-1}/m)$
\item[$\cdot$]$D_1 . D_2$ : the intersection number of two divisors $D_1$ and $D_2$ on a surface
\item[$\cdot$]$D^2$ : the self-intersection number of a divisor $D$ on a surface
\item[$\cdot$]$(X_0: \cdots : X_m)$ : a set of homogeneous coordinates on ${\mathbb P}^m$
\item[$\cdot$]$g(C)$ : the genus of a smooth curve $C$
\item[$\cdot$]Supp $D$ : the support of a divisor $D$ 
\item[$\cdot$]$X_{(4)}$ : a smooth quartic surface in ${\mathbb P}^3$
\item[$\cdot$]$X_{(23)}$ : a smooth $(2,3)$-complete intersection of hypersurfaces in ${\mathbb P}^4$ 
\item[$\cdot$]$X_{(222)}$ : a smooth $(2,2,2)$-complete intersection of hypersurfaces in ${\mathbb P}^5$
\end{namelist}

\medskip 

Suppose that $S$ is a $K3$ surface such that $(S,D)$ defines a Galois embedding with the Galois 
group $G \subset {\rm Aut}(S)$. 
Let $\omega_S$ be a nowhere vanishing holomorphic two form of $S$
Then, let $\varepsilon : G \longrightarrow {\mathbb C}^{\times}=\mathbb C\setminus \{0\}$ be the character of the 
natural representation of $G$ on the space $H^{2,0}(S)={\mathbb C}{\omega_S}$, i.e., $\varepsilon(\sigma)=\lambda$ for 
$\sigma \in G$ if ${\sigma}^*(\omega_S)=\lambda \omega_S$.  
There exists a multiplicative group $\Gamma_m$ of the $m$-th roots of unity and the following exact sequence of
 groups:
$$
1 \longrightarrow G_s \longrightarrow G \stackrel{\varepsilon}{\longrightarrow} {\Gamma_m} \longrightarrow 1,  \eqno(6)
$$
where $G_s$ is a symplectic group \cite{mu}. Let $\pi : S \longrightarrow {\mathbb P}^2$ be the projection, which is a Galois covering 
defined in Section 2. 
Let $W$ be the center of the projection and $H$ a general hyperplane containing $W$. Put $C = S \cap H$. 
Then $C$ is an irreducible smooth curve and $C \sim D$. 

\begin{lemma}\label{20}
The representation $r : G \longrightarrow {\rm Aut}(C)$ given by $r(\sigma)=\sigma|_{C}$ is injective. 
\end{lemma}

\begin{proof}
Note that $\sigma \in G$ is an automorphism of $S$ over $\mathbb P^2$ and $\sigma(C)=C$. If $\sigma|_{C}$ is identity, 
then $C$ is a component of the ramification divisor of the covering. Since $C$ is given by $H$ which is general, $\sigma$ 
must be identity.
\end{proof}

The restriction $\pi|_{C} : C \longrightarrow {\mathbb P}^1$ turns out to be a Galois covering, where the Galois group is isomorphic to $G$. 
Since ${\mathop{\mathrm{H^1}}\nolimits}(S,\ {\mathcal O})=0$ and the canonical divisor on $S$ is trivial, the restriction of $f_D$ to $C$ 
gives the canonical embedding of $C$. 
Therefore $C$ has a Galois embedding given by its canonical divisor.  

\medskip

\begin{lemma}\label{22}
The group $G$ is non-symplectic, i.e., $\Gamma_m \ne \{1 \}$. 
\end{lemma}

\begin{proof}
Suppose $\Gamma_m = \{1 \}$. Then, $G=G_s$. This means that the fixed loci of each element of $G$ is 
at most finitely many points. This contradicts to Proposition \ref{15}.  
\end{proof}

Let $R$ be the ramification divisor for $\pi$.  

\begin{lemma}\label{17}
We have $R \sim 3D$ and Supp $R$ is connected. Each irreducible component of Supp $R$ is smooth. 
\end{lemma}

\begin{proof}
Since the canonical divisor on $S$ is trivial, using the adjunction formula, we get $\pi^*(-3\ell)+R \sim 0$ for a line $\ell$ 
in $\mathbb P^2$. Since $\pi^*(\ell) \sim D$, we have $R \sim 3D$, hence $R$ is very ample. 
\end{proof}

\begin{example}\label{48}
Let $S$ be the Fermat quartic surface: $X_0^4+X_1^4+X_2^4+X_3^4=0$ and $P$ be one of the points $(1:0:0:0), (0:1:0:0), (0:0:1:0)$ and  
$(0:0:0:1)$. The projection from $P$ to the hyperplane $\mathbb P^2$ defines a cyclic Galois covering (such $P$ is called a 
Galois point \cite{y3}). Note that $S$ is a Kummer surface $Km(E \times E)$, where $E=\mathbb C/(1, e_4)$. Further, it is 
a singular $K3$ surface, i.e., $\rho(S)=20$ (cf. \cite{in}).
\end{example}

\section{abelian case}
In the case of Galois embeddings of abelian surfaces, the group cannot be abelian.  
However, in the case of $K3$ surfaces, the group can be a cyclic group as in Example \ref{48}. 
Nikulin \cite{ni} shows that there exist many abelian automorphism groups for $K3$ surfaces. 
So let us consider the Galois embedding where the Galois group $G$ is abelian.  
Hereafter we assume $G$ is abelian if not otherwise mentioned.   
  
\begin{theorem}\label{51}
If the Galois group $G$ is abelian, then $G \cong Z_4$, $Z_6$ or ${Z_2}^3=Z_2 \times Z_2 \times Z_2$ and $S$ is isomorphic to $S_{(4)}$, $S_{(23)}$ or $S_{(222)}$ 
 respectively.
\end{theorem}

We will give concrete examples for the three surfaces in Section 5.  
We note that for the proof of Theorem \ref{51} we do not use the property of Galois embedding, but only use that  
the covering $\pi:S \longrightarrow \mathbb P^2$ is Galois (except in the proof of Claim \ref{49}). 
So the result may be known, but for the sake of completeness, we give the proof in this article. 

Before proceeding with the proof, we fix the notation. 
Let $\pi : S \longrightarrow \mathbb P^2$ be the Galois covering induced by the projection. Put $|G|=n$ and assume that   
$R=(n_1-1)C_1 + \cdots + (n_r-1)C_r$, where $C_i$ are irreducible components.  
For $\sigma \in G$ put $F(\sigma)=\{ \ x \in S \ | \ \sigma(x)=x \  \}$. 

\begin{lemma}\label{36}
For each point $x \in \mathrm{Supp}\ R$ the stabilizer of the point $G_x=\{\ \sigma \in G \ | \ \sigma(x)=x \  \}$ is generated by at most two elements. 
\end{lemma}

\begin{proof}
There exists an open neighbourhood $U_x$ and coordinates on it such that $G_x$ has a representation in $GL(2, \mathbb C)$. 
Since $G$ is abelian, we can assume each element of $G_x$ is generated by one or two diagonal matrices $[\alpha, 1]$ and $[1, \beta]$, 
where $\alpha^n=\beta^n=1$.  
\end{proof}

\begin{lemma}\label{37}
The following assertions hold true.
\begin{enumerate}
\item Supp $R$ is connected.
\item Each irreducible component is a smooth curve.
\item Supp $R$ has normal crossings. 
\end{enumerate}  
\end{lemma}

\begin{proof}
Since $R \sim 3D$, we have $R$ is ample, hence Supp $R$ is connected. 
Each component $C_i$ is given by $\{ \ x \in S \ | \ \sigma(x)=x \  \}$ for some $\sigma \in G \setminus \{id \}$, 
where locally $\sigma$ can be expressed as a diagonal matrix $[\alpha, 1]$. Thus $C_i$ is smooth. 
Suppose $x \in $Supp $R$ is an intersection point of some components $C_i$ ($1 \le i \le r$).  
As we have seen in the proof of Lemma \ref{36}, there exist two elements $[\alpha,1]$ and $[1, \beta]$ which are generators 
of the stabilizer $G_x$. Thus there exist just two irreducible components meeting normally.  
\end{proof}

\begin{lemma}\label{38}
For each irreducible component $C$ of $R$, we have Supp $\pi^*(\pi(C))$ $=C$, i.e., $\tau(C)=C$ 
for any $\tau \in G$. In particular $C$ is an ample divisor.  
\end{lemma}

\begin{proof}
Let $\sigma \in G$ satisfy $\sigma \ne id$ and $\sigma|_C=id$. 
Since $\pi(C)$ is ample and $\pi : S \longrightarrow \mathbb P^2$ is a finite morphism, $\pi^*(\pi(C))$ is ample and hence Supp $\pi^*(\pi(C))$ is connected. 
Suppose Supp $\pi^*(\pi(C))$ is reducible. Then, there exists another irreducible component $C'$ of $\pi^*(\pi(C))$ 
such that $C'=\sigma'(C)$ for some $\sigma' \in G$ and $C \cap C' \ne \emptyset$. Since $\sigma \sigma'=\sigma'\sigma$, 
we have $\sigma(\sigma'(y))=\sigma'(y)$ for any $y \in C$. This means $\sigma|_{C'}=id$. 
Take $x \in C \cap C'$. Then $C$ and $C'$ have a normal crossing at $x$ by Lemma \ref{37}.  However, looking at $\sigma$ near $x$, the $\sigma$ 
can be expressed as one of the diagonal matrices $[\alpha, 1]$ and $[1, \beta]$, where $\alpha \ne 1$ and $\beta \ne 1$. This contradicts to that $\sigma|_{C'}=id$. 
\end{proof}

\begin{corollary}\label{39}
With the same notation as in Lemma \ref{38}, we have $C^2>0$, hence $g(C) \geq 2$.  
\end{corollary}

\begin{proof}
Since $\pi^*(\pi(C))$ can be expressed as $mC$, we have $m^2C^2=n(\pi(C))^2 \geq n$. 
Since $2g(C)-2=C^2$ we have the assertion. 
\end{proof}

Put $G_i=\{\ \sigma \in G \ | \ \sigma|_{C_i}=id \  \}$ ($1 \le i \le r$). 
Then $G_i$ is determined uniquely by $C_i$ and not a trivial subgroup of $G$. 

\begin{lemma}\label{40}
The group $G_i$ is cyclic.
\end{lemma}

\begin{proof}
For a general point $x \in C_i$, taking a suitable local coordinates, we can express each $\sigma \in G_i$ 
as $[\alpha, 1]$. We have a monomorphism $\rho : G_i \longrightarrow \mathbb C^{\times}$, where 
$\rho(\sigma)=\alpha$. Since $\rho(G_i)$ is a cyclic group, so is $G_i$. 
\end{proof}

\begin{lemma}\label{43}
The surface $S_i=S/G_i$ $(1 \le i \le r)$ is a smooth rational surface.
\end{lemma}

\begin{proof}
Since near each point $x \in C_i$, the $\sigma \in G_i$ can be expressed as a diagonal matrix. Hence $S_i$ is smooth. 
Let $K_i$ be a canonical divisor on $S_i$. Then, we have $\pi_i^*(K_i)+R_i \sim 0$, where 
$\pi_i : S \longrightarrow S_i$ and $R_i$ is a ramification divisor for $\pi_i$. 
Since $R_i$ is effective, we infer that $\dim {\mathop{\mathrm{H^0}}\nolimits}(S_i, {\mathcal O}(2K_i))=0$. 
Clearly we have $\dim {\mathop{\mathrm{H^0}}\nolimits}(S_i, \Omega_i^1)=0$, where $\Omega_i^1$ is the sheaf of 
holomorphic 1-forms on $S_i$. Therefore $S_i$ is rational by Castelnuovo's Rationality Criterion.   
\end{proof}

\begin{lemma}\label{35}
There does not exist $\tau \in G$ such that $\tau \ne id$ and $F(\tau)=\emptyset$.  
\end{lemma}

\begin{proof}
Suppose otherwise. Since $G$ is abelian, expressing $G \cong \langle \tau \rangle \times G'$, we put $S'=S/G'$. 
Then $S'$ is a smooth rational surface. Because, as we see in the proof of Lemma \ref{36}, $G_x$ is generated locally by 
reflections. Hence $S'$ is smooth. Since there exists a covering $S_i \longrightarrow S'$ (or $S_i=S'$) and $S_i$ is rational, 
we see that $S'$ is rational.
The $\pi':S' \longrightarrow \mathbb P^2=S/G'$ is an unramified double covering, this is a contradiction.  
\end{proof}

\begin{lemma}\label{41}
The group $G_i$ $(1 \le i \le r)$ determines $C_i$ uniquely and $G_i \cap G_j$ consists of identity if $i \ne j$. 
Therefore, there exists a one to one correspondence between the set $\{ \ G_i \ | \ 1 \le i \le r \}$ and $\{ \ C_i \ | \ 1 \le i \le r \ \}$. 
\end{lemma}

\begin{proof}
If $C_i \ne C_j$, then we have $C_i \cap C_j \ne \emptyset$ by Lemma \ref{38}. 
Take $x \in C_i \cap C_l$. Consider $G_i$ and $G_j$ in a neighbourhood of $x$. 
Since $G$ is abelian, there exist generators $[\alpha, 1]$ and $[1, \beta]$ of $G_i$ and $G_j$ respectively. 
If $\sigma \in G_i \cap G_j$, then $\sigma|_{C_i}=\sigma|_{C_j}=id$. This implies that $\sigma=id$.    
\end{proof}

\begin{lemma}\label{42}
The group $G$ can be expressed as a direct product $G_1 \times \cdots \times G_r$, where each $G_i$ is cyclic $(1 \le i \le r)$. 
\end{lemma}

\begin{proof}
For each element $\sigma \in G$, there exists a fixed point of $\sigma$ by Lemma \ref{35}. 
If $F(\sigma)$ contains a curve, then there exist $i$ such that $\sigma|_{C_i}=id$. This means that $\sigma \in G_i$. 
On the other hand, if $F(\sigma)$ consists of only points, then take $x \in F(\sigma)$. 
It is easy to see that there exist two curves $C_i$ and $C_j$ containing $x$. 
Then $\sigma$ can be expressed as a product of elements of $G_i$ and $G_j$. 
Therefore, we conclude the assertion from Lemma \ref{41}.   
\end{proof}

\medskip

Let $\Delta_i$ be the plane curve $\pi(C_i)$ and put $\Delta=\Delta_1+\cdots+\Delta_r$. 

\begin{lemma}\label{44}
Each $\Delta_i$ is smooth $(1 \le i \le r)$ and $\Delta$ has normal crossings.
\end{lemma}

\begin{proof}
In the proof of Lemma \ref{38}, we have shown that $\tau(C_i)=C_i$ for each $\tau \in G$. Therefore $G$ acts on $C_i$ 
and we can consider $C_i/G$. We denote it by $\Delta_i$. Hence $\Delta_i$ is smooth. 
For a point $x \in C_i \cap C_j$ we have $\sigma_i(x)=\sigma_j(x)=x$ and $\sigma_k(C_i)=C_i$ and $\sigma_k(C_j)=C_j$ 
$(1 \le i,j,k \le r)$. Hence $\Delta$ has normal crossings. 
\end{proof}

Put $n_i=|G_i|$. Then we have $n=\prod_{i=1}^r n_i$ by Lemma \ref{42}. 
Denote by $\chi(V)$ the topological Euler characteristic of a curve or a surface $V$. 

\begin{lemma}\label{45}
We have $n_i=2, 3$ or $4$ for each $i$. 
\end{lemma}

\begin{proof}
Put $\bar{C_i}=\pi_i(C_i)$ where $\pi_i : S \longrightarrow S_i=S/G_i$. 
Compare $\chi(S)$ and $\chi(S_i)$.  
Since $G_i=\langle \sigma_i \rangle$ and $\sigma_i|_{C_i}=id$, the $C_i$ is isomorphic to $\bar{C_i}$. 
Hence we have 
\[
\begin{array}{ccl}
\chi(S) & = & \chi(S-C_i)+\chi(C_i) \\
        & = & n_i\chi(S_i-\bar{C_i}) + \chi(\bar{C_i}) \\
        & = & n_i\chi(S_i)+(1-n_i)\chi(\bar{C_i})
\end{array}
\]
We have $\chi(C_i)=2-2g(C_i)=\chi(\bar{C_i})$. 
Therefore we have 
\[
24=n_i\chi(S_i)+(n_i-1)(2g(C_i)-2). 
\eqno{(7)}
\] 
Since $S_i$ is a smooth rational surface by Lemma \ref{43}, we have $\chi(S_i) \geq 3$. 
Further, we have $g(C_i) \geq 2$ by Corollary \ref{39}. 
Thus, clearly we have $n_i \le 5$. In case $n_i=5$, we have 
$24=5\chi(S_i)+8(g(C_i)-1)$, but this cannot hold. 
Whence we conclude $n_i \le 4$. 
\end{proof}

Next we consider the branch divisor for $\pi$. 
Put $d_i=\deg \Delta_i \ (1 \le i \le r)$. 

\begin{lemma}\label{46}
We have the equality 
\[
d_1\left(1-\frac{1}{n_1} \right)+ \cdots + d_r \left(1-\frac{1}{n_r}  \right)=3.    
\eqno{(8)}
\]
In particular, we have $r \le 6$. 
\end{lemma}

\begin{proof}
Letting $\ell$ be a line in $\mathbb P^2$, we have $\pi^*(3\ell).\pi^*(\ell)=3n$. 
Since $\pi^*(\Delta_i)=n_iC_i$, we have $n_iC_i.\Gamma = \pi^*(\Delta_i).\Gamma=nd_i$, where 
$\Gamma=\pi^*(\ell)$. 
Since $R \sim 3D$ by Lemma \ref{17} and $D \sim \pi^*(\ell)$, we get
\[
(n_1-1)\Gamma C_1 + \cdots + (n_r-1)\Gamma C_r=3n,  
\]
hence 
\[
(n_1-1)\frac{d_1}{n_1}n + \cdots + (n_r-1)\frac{d_r}{n_r}=3n.  
\] 
This proves the equation. Since $n_i \geq 2$ and $d_i \geq 1$, we have $r \le 6$.  
\end{proof}
  
\begin{lemma}\label{47}
We have $d_i \geq 2$ for each $i$. 
\end{lemma}

\begin{proof}
Put $\hat{G_i}=G/G_i$ and consider the coverings  
\[
p_i : S \longrightarrow S/\hat{G_i}=\hat{S_i} \ \ \mathrm{and} \ \ 
q_i : \hat{S_i} \longrightarrow \mathbb P^2=S/G. 
\] 
By Lemma \ref{38} $\hat{G_i}$ acts on $C_i$, hence put $\hat{C_i}=p_i(C_i)=C_i/\hat{G_i}$. 
By repeating the similar arguments as in the proof of Lemma \ref{43}, we conclude $\hat{S_i}$ is a smooth rational surface. 
Suppose $d_i=1$, Then, $q_i(\hat{C_i})=\Delta_i$ is a line $\ell$. 
Hence we get $q_i^*(\ell)=n_i\hat{C_i}$. This means $q_i^*(\ell)^2=n_i=n_i^2\hat{C_i}^2$. 
Hence $n_i \hat{C_i}^2=1$, i.e., $n_i=1$. 
This is a contradiction. 
\end{proof}

Making use of Lemmas \ref{45}, \ref{46} and \ref{47}, we determine $r, n_i$ and $d_i$ $(1 \le i \le r)$. 

\begin{claim}\label{56}
If there exists $i$ such that $n_i=4$, then $G \cong Z_4$.
\end{claim}

\begin{proof} 
We may assume $i=1$. 
We prove the claim by examining the cases: 

\begin{enumerate}
\item[(i)] $r=1$. 
\item[(ii)] There exists $j \geq 2$ such that $n_j=4$. 
\item[(iii)] $n_j \le 3$ for all $j \geq 2$ and 
there exists for some $j \geq 2$ such that $n_j=3$ 
\item[(iv)] $n_j=2$ for all $j \geq 2$  
\end{enumerate}

In case of (i) we observe the equality (7).  
We have $24=4\chi(S_1)+6(g(C_1)-1)$. Since $g_1=g(C_1) \geq 2$, we have $g_1=3$ and $\chi(S_1)=3$. 
Since $S_1$ is a smooth rational surface, we have $S_1 \cong \mathbb P^2$, hence $d_1=4$.  
Therefore we have $G \cong Z_4$. 

The case (ii) does not occur. Suppose otherwise. Then, since $d_j \geq 2$, we infer from (8) that $r=2$, $d_1=d_2=2$ and $n_1=n_2=4$. 
Note that $\sigma_1|_{C_1}=id$ and $\sigma_1$ acts on $C_2$. 
Thus we have $\chi(\Delta_i)=2$ and $\chi(C_i)=-4$ $(i=1, 2)$. 
Then we get 
\[
\begin{array}{ccl}
\chi(S) & = & \chi(S-(C_1 \cup C_2))+\chi(C_1 \cup C_2) \\
        & = & 16 \chi(\mathbb P^2-(\Delta_1 \cup \Delta_2)) + \chi(\Delta_1)+\chi(\Delta_2)-\chi(\Delta_1 \cap \Delta_2) \\
        & = & 36. 
\end{array}
\]
which is a contradiction. 

The case (iii) does not occur. Suppose otherwise. Then, from (8) we have 
\[
3 \geq d_1 \left(1 - \frac{1}{4}\right)+ d_2 \left(1-\frac{1}{3} \right) \geq  \frac{3}{2}+\frac{2}{3}d_j. 
\]
Since $d_j \ne 1$, we have $d_j=2$. This means that $r=2$ and $3d_2/4 = 5/3$, which is a contradiction.
 
The case (iv) does not  occur. Suppose otherwise. Then, from (8) we have that $3d_1+2(d_2 + \cdots + d_r)=12$, which implies that $r=2$ and $d_2=3$, i.e., $n_1=4, d_1=2$ and $n_2=2, d_2=3$. 
Note that $\sigma_1|_{C_1}=id$ and $\sigma_1$ acts on $C_2$. We infer readily that $C_2':=C_2/G$ is a smooth curve in $S_1:=S/G_1 \cong \mathbb P^2$. 
We have $\pi=q_1 \cdot p_1$, where $p_1 : S \longrightarrow S_1$ and $q_1 : S_1 \longrightarrow S/G \cong \mathbb P^2$. 
Then $q_1 : S_1 \longrightarrow \mathbb P^2$ is a double covering branched along just $\Delta_2$, which is cubic. 
This is a contradiction. 
\end{proof}

Therefore we assume $n_i=2$ or $3$. Since $d_i \geq 2$, we have $r \le 3$. 
\begin{enumerate}
\item In case $r=3$, it is easy to see that $d_i=n_i=2$ for $i=1, 2, 3$. Then, $G \cong {Z_2}^3$. 
\item In case $r=2$, we have 
\[
d_1\left(1- \frac{1}{n_1}\right)+d_2\left(1-\frac{1}{n_2}\right)=3. 
\]
Then we have 
$5 \le d_1+d_2 \le 6$. We assume $d_1 \geq d_2$ and find the solutions.  
Here we use the notation $(a,b;c,d)$, which means $a=d_1, b=n_1$ and $c=d_2, d=n_2$. 

(b-1) In the case $d_1+d_2=6$, we have $(4,2;2,2)$ or $(3,2;3,2)$.  

(b-2) In the case $d_1+d_2=5$, we have $(3,3;2,2)$ or $(3,2;2,4)$.  
\end{enumerate}

\begin{claim}\label{49}
The case $r=2$ and $n_1=n_2=2$ does not occur.
\end{claim}

\begin{proof}
We show that $G$ cannot be isomorphic to $Z_2 \times Z_2$. 
Suppose $(S, D)$ gives a Galois embedding. Then, $|G|=D^2=4$ and $\dim {\mathop{\mathrm{H^0}}\nolimits}(S,\ {\mathcal O(D)}=4$. 
Thus $f_D(S)$ is a quartic surface in $\mathbb P^3$. By Corollary \ref{a4} the Galois group must be 
cyclic, which is a contradiction. 
\end{proof}

\begin{claim}\label{}
The case $(3,2;2,4)$ does not occur.
\end{claim}

\begin{proof}
Suppose otherwise. Then, there exists a smooth surface $S_2=S/G_2$, which is a double covering of 
$S/G \cong \mathbb P^2$ branched along $\Delta_1$. 
However, $\deg \Delta_1$ is odd, hence the double covering cannot exist. This is a contradiction. 
\end{proof}

Thus only the case $n_1=3$ and $n_2=2$ remains, which corresponds to $G \cong Z_3 \times Z_2 \cong Z_6$. 
Combining the results above, we complete the proof. The last assertion $S \cong S_{(222)}$ will be proved in 
Theorem \ref{11} below.

\section{particulars}
In this section we describe all the surfaces in Theorem \ref{51}. 
We study the Galois embeddings of $S$ in detail for $D^2=2m$, where $m=2,3$ and $4$. 
Let $C$ be a general member of the complete linear system $|D|$. Then we have $g(C)=m+1$, where $g=g(C)$ is the genus of $C$ 
and $\dim \mathrm{H}^0(S,{\mathcal O}(D))=g+1$.
So $f_D(S)$ is assumed to be embedded in $\mathbb P^g$. 

\bigskip

\underline{CASE 1. $g=3$}

Assume $|G|=D^2=4$. Then $G \cong Z_4$. 
Taking suitable homogeneous coordinates such that the Galois point be $X_0=X_1=X_2=0$, then 
$\beta_1(\sigma)$, which is the projective transformation (3) defined in Section 2, can be expressed as a diagonal matrix $[1, 1, 1, e_4]$. 
Since the defining equation of $f_D(S)$ is invariant by this transformation, we infer readily the following. 

\begin{theorem}\label{52}
We have that $g=3$ if and only if $G \cong Z_4$. In this case the defining equation of $f_D(S)$ can be given by 
$X_3^4+F_4(X_0, X_1, X_2)=0$, where $F_4(X_0, X_1, X_2)$ is a form of degree four. 
\end{theorem}

\begin{remark}\label{113}
The maximal number of Galois points for the surface in Theorem \ref{52} is four. 
And it is four if and only if it is the Fermat quartic (Example \ref{48}). 
\end{remark}

We can show a relation between the possibility of Galois embedding and the Picard number $\rho(S)$ for $S$. 

\begin{lemma}\label{54}
For the surface $S$ in Theorem \ref{52} we have $\rho(S) \geq 2 $. 
\end{lemma}

Before proceeding with the proof we note the following.

\begin{remark}\label{55}
A smooth quartic plane curve $\Delta$ has at least 16 bitangent lines.
\end{remark}

\begin{proof}
Let $\hat{\Delta}$ be the dual curve of $\Delta$. 
Then we have $\deg \hat{\Delta}=12$ and the genus of smooth model of $\hat{\Delta}$ is 3. 
Let $T_P$ be the tangent line to $\Delta$ at $P$. If the intersection number of $T_P$ and $\Delta$ at 
$P$ is $i+2 \geq 3$, then $P$ is said to be an $i$-flex. 
Letting $a_i$ be the number of $i$-flexes of $\Delta$ $(i=1, 2)$, we see that $\hat{\Delta}$ has 
$a_1$-pieces of $(2, 3)$ cusps and $a_2$-pieces of $(3, 4)$ cusps. 
Referring to \cite[Theorem 6.11]{ii}, we get $a_1+2a_2=24$. 
If $b$ is the number of nodes of $\hat{\Delta}$, then, applying the genus formula \cite[Theorem 9.1]{ii}, 
we get $b \geq 16$. Since a node of $\hat{\Delta}$ corresponds to a bitangent line of $\Delta$, the proof is 
complete. 
\end{proof}

Let $\Delta$ be the branch locus of $\pi:S \longrightarrow \mathbb P^2$, which is a smooth quartic curve. 
Let $\ell$ be a bitangent line to $\Delta$ and we consider $\pi^*(\ell)$.  

\begin{claim}
The curve $\pi^*(\ell)=\Gamma$ is a sum of two $(-2)$-curves.
\end{claim}

\begin{proof}
Let $P_1$ and $P_2$ be $\pi^{-1}(\ell \cap \Delta)$. 
Suppose $\Gamma$ is irreducible. Then, it is not difficult to see by local consideration that it has two singular points $P_i \ (i=1,2)$.  
Each $P_i$ is locally isomorphic to the singularity defined by $y^2=x^4$. 
Let $\mu:\widetilde{\Gamma} \longrightarrow \Gamma$ be the resolution of singularities. 
Then $\pi|_{\Gamma} \cdot \mu:\widetilde{\Gamma} \longrightarrow \ell$ is a cyclic Galois covering of degree 4. 
Then, by Riemann-Huwitz formula we have $2g(\widetilde{\Gamma})-2=4(-2)+4=-4$, which is a contradiction. 
Hence we have $\pi^*(\ell)=\Gamma_1+\Gamma_2$, where $\Gamma_i$ ($i=1,2$) is a $(-2)$-curve and $\Gamma_1.\Gamma_2=4$.   
\end{proof}

From this claim Lemma \ref{54} is clear. 

\bigskip

\underline{CASE 2. $g=4$}

In this case $|G|=6$. So $G \cong Z_6$ and put $G=\langle \sigma \rangle$.  

\begin{theorem}\label{53}
If $G \cong Z_6$, then $f_D(S)$ is a $(2,3)$-complete intersection, furthermore the defining equation of $f_D(S)$ can be given by 
$F_2(X_0, X_1, X_2)+{X_3}^2=F_3(X_0, X_1, X_2)+{X_4}^3=0$, where $F_i(X_0, X_1, X_2)$ is a form of $X_0, X_1, X_2$ with degree $i$ $(i=2,3)$ 
such that each curve $F_i=0$ in $\mathbb P^2$ has no singular points.  
\end{theorem}

\begin{proof}
Since the embedding is given by $|D|$, where $D^2=6$, the surface $f_D(S)$ is the smooth complete intersection. 
In the proof of Theorem \ref{51}, in case $|G|=6$, we have shown that $d_1=n_1=3$ and $d_2=n_2=2$. 
So that $\sigma^2$ (resp. $\sigma^3$) is identity on $C_1$ (resp. $C_2$). 
We have two covering maps $f_i:S_i:=S/\langle {\sigma}^i \rangle \longrightarrow \mathbb P^2$, where $i=2$ and $3$. 
The $f_i$ is a Galois covering of degree $i$ branched along $\Delta_i$. 
Put $g_i:S \longrightarrow S_i$. Then we have $f_2g_2=f_3g_3$. 
Since $\Delta_2$ and $\Delta_3$ have normal crossings, the fiber product $S_2 \times_{\mathbb P^2} S_3$ is smooth. 
Since $S$ is also given by the double covering of $S_3$ branched along $g_3(C_2)$, we see that 
$S$ is isomorphic to the fiber product $S_2 \times_{\mathbb P^2} S_3$. 
Furthermore, by taking a suitable coordinates on $\mathbb P^2$, we can assume 
$S_2$ is defined by $X_3^2+F_2(X_0,X_1,X_2)=0$ and $S_3$ by $X_4^3+F_3(X_0,X_1,X_2)=0$. 
This proves the theorem. 
\end{proof}

There is some relation between a Galois embedding and the trivilality of the symplectic group $G_s$.  
From the following Corollary \ref{110} to Corollary \ref{26} 

\noindent \underline{we do not assume that $G$ 
is abelian}. 

\begin{corollary}\label{110}
Suppose $S$ has a Galois embedding. Then, $G_s$ is trivial if and only if the embedding is given by a divisor $D$ such that $D^2=4$ or $6$. 
\end{corollary}

\begin{proof}
If $G_s$ is trivial, then $G$ is cyclic, hence $G \cong Z_4$ or $Z_6$. Conversely, if $G \cong Z_4$ or $Z_6$, then the defining 
ideal of $f_D(S)$ and the generator $\sigma$ are given in Theorems \ref{52} or \ref{53}.  
Referring to \cite[Lemma 2.1]{mu}, we conclude $|\Gamma_m|=4$ or $6$, where $\Gamma_m$ is the cyclic group in (6). 
Thus $G_s$ is trivial. 
\end{proof}

We consider the Picard number $\rho(S)$ for the surfaces $S$ in Theorem \ref{52} and \ref{53}. 

\begin{lemma}\label{23}
If $S$ is the surface in Theorem \ref{53}, then $\rho(S) \geq 2$. 
\end{lemma}

\begin{proof}
Let $\sigma$ be a generator of $G$ and consider $S/\langle {\sigma}^3 
\rangle$, which is a rational surface containing $(-1)$-curve. In fact, it is a smooth cubic in ${\mathbb P}^3$. Then we 
infer readily that $S$ has a $(-2)$-curve. 
\end{proof} 

Let $T_X$ be the transcendental lattice for a $K3$ surface $X$. 
Machida and Oguiso \cite{mo} prove the following: 

\begin{lemma}\label{24}
Let $X$ be a $K3$ surface and $G$ be a finite automorphism group of $X$. 
Assume that {\rm rank} $T_X \geq 14$. Then $G_s=\{ 1 \}$, or equivalently, $G \cong \Gamma_m$. 
\end{lemma}

As we expect, a ``general $K3$ surface" does not have a Galois embedding. Indeed, combining the results above, 
we deduce the following assertion. 

\begin{theorem}\label{25}
If $\rho(S)=1 $, then $S$ has no Galois embeddings.  
\end{theorem}

\begin{corollary}\label{26}
If $S$ has a Galois embedding and $\rho(S) \le 8$, then it is isomorphic to $S_{(4)}$ or $S_{(23)}$. 
Hence $G \cong Z_4$ or $Z_6$. 
\end{corollary} 

Then, what can we say about a Galois embedding when $\rho(S)$ is large ? Can we say that $S$ has the Galois embedding in 
the case where $\rho(S)$ is the maximal possible $20$ ? 

\medskip

\underline{CASE 3. $g=5$}

In this case $|G|=8$. So $G \cong Z_2^3$.  

\begin{theorem}\label{11}
If $G \cong Z_2^3$, then $S$ is a double covering of $S_{(22)}$,
 where $S_{(22)}$ is a rational surface of $(2,2)$-complete intersection in ${\mathbb P}^4$. 
Furthermore we have the following sequence of surfaces: 
$$
S \stackrel{\pi_1}\longrightarrow S_{(22)} \stackrel{\pi_2}\longrightarrow S_{(2)} \stackrel{\pi_3}{\longrightarrow} {\mathbb P}^2, \eqno(6)
$$
which have the following properties.
\begin{enumerate}
\item $\pi_i$ ($i=1,2,3$) is a double covering and $\pi=\pi_3 \cdot \pi_2 \cdot \pi_1$. 
\item $S_{(22)}$ is a surface of $(2,2)$-complete intersection of ${\mathbb P}^4$. 
\item $S_{(2)}$ is a smooth conic in ${\mathbb P}^3$.
\end{enumerate} 
Further more, the defining equation of $f_D(S)$ can be given by 
$X_3^2+F_{23}(X_0, X_1, X_2)=X_4^2 +F_{24}(X_0, X_1, X_2)=X_5^2 +F_{25}(X_0, X_1, X_2)=0$, where $F_{2i}(X_0, X_1, X_2)$ is a form of $X_0, X_1, X_2$ with degree $2$, such that each curve $F_{2i}=0$ $(i=3,4,5)$ in $\mathbb P^2$ has no singular points. 
In particular $S$ is isomorphic to $S_{(222)}$. 
\end{theorem}

\begin{proof}
The proof is done by the same way as the one of Theorem \ref{53}. 
In this case we have $d_i=n_i=2$ $(i=1,2,3)$. Let $G=\langle \sigma_1, \sigma_2, \sigma_3 \rangle$ and $G_i=\langle \sigma_i \rangle$. 
 Put $S_i=S/\langle \sigma_i \rangle$ and $S_{ij}=S/\langle \sigma_i, \sigma_j \rangle$, where $i \ne j$. 
Then $S$ is a double covering of $S_i$ and so is $S_i$ of $S_{ij}$, and $S_{ij}$ is a double covering of $\mathbb P^2$ 
branched along $\Delta_k$, where $i,j,k$ are mutually distinct. 
It is easy to see that $S_i$ is isomorphic to the fiber product $S_{ij} \times_{\mathbb P^2} S_{ik}$ and hence 
$S$ is isomorphic to $(S_{12} \times_{S_1} S_{13}) \times_{\mathbb P^2} S_{23}$. 
In particular $S$ is isomorphic to $S_{(222)}$. 
\end{proof}

\begin{remark}
In the case where $G$ is not abelian and $g=4$ or $g=5$, we can show that $G$ is isomorphic to 
the dihedral group. Furthermore such a $K3$ is obtained as a Galois closure surface of some rational 
surface.  
The research for non-abelian case will be done in the forthcoming paper. 
\end{remark}

There are a lot of problems concerning our theme, we pick up some of them.

\bigskip

\noindent{\bf Problems.}

\begin{enumerate}
\item How many Galois subspaces do there exist for one Galois embedding and how is their arrangement?  
In the case of a smooth quartic surface in $\mathbb P^3$, see Remark \ref{113}. Then, how is the case for $(2,3)$-complete intersection or 
$(2,2,2)$-complete intersection?   
\item Does there exist a $K3$ surface $S$ on which there exist two divisors $D_i$ ($i=1, 2$) such that they  
give Galois embeddings and $D_1^2 \ne D_2^2$?  
\item Does each singular $K3$ surface have a Galois embedding ?
\end{enumerate}

\bigskip

\begin{center}
{\bf Acknowledgement}
\end{center}
The author expresses his gratitude to Hiroyasu Tsuchihashi, who gave him useful information.

\bibliographystyle{amsplain}

\end{document}